\documentclass[11pt]{article}
\usepackage{amsmath,amssymb,amsthm,graphicx,enumerate,bbm,natbib}
\usepackage{color}
\def\pr{\mathbb P}
\def\esp{\mathbb E}

\newtheorem{theo}{Theorem}
\newtheorem{lem}[theo]{Lemma}
\newtheorem{coro}[theo]{Corollary}
\newtheorem{prop}[theo]{Proposition}
\newtheorem{defi}{Definition}
\newtheorem{hypo}{Assumption}
\theoremstyle{remark}

\newtheorem{xmpl}{Example}

\begin{document}

\title{Limit conditional distributions for bivariate vectors\\
    with polar representation}

\author{Anne-Laure Foug\`eres\thanks{Universit\'e de Lyon,  Universit\'e Lyon 1,  CNRS, UMR5208, Institut Camille Jordan, 43 blvd du 11 novembre 1918, F-69622 Villeurbanne-Cedex, France} \and
  Philippe Soulier\thanks{Universit\'e Paris Ouest-Nanterre, Equipe Modal'X, 200 av de la r\'epublique, F-92000 Nanterre, France } }

\maketitle

\begin{abstract}
  We investigate conditions for the existence of the limiting
  conditional distribution of a bivariate random vector when one
  component becomes large. We revisit the existing literature on the
  topic, and present some new sufficient conditions.  We concentrate
  on the case where the conditioning variable belongs to the maximum
  domain of attraction of the Gumbel law, and we study geometric
  conditions on the joint distribution of the vector.  We show that
  these conditions are of a local nature and imply asymptotic
  independence when both variables belong to the domain of attraction
  of an extreme value distribution.  The new model we introduce can
  also be useful to simulate bivariate random vectors with a given limiting conditional  distribution.

\end{abstract}

\noindent {\bf Keywords:} 
Conditional excess probability; conditional extreme-value mo\-del;
$\Gamma$-varying tail; asymptotic independence; elliptic
distributions; second order correction.

\section{Introduction}

In many practical situations, there is a need of modeling multivariate
extreme events. Extreme means, roughly speaking, that no observations
are available in the domain of interest, and that extrapolations are
needed. Multivariate extreme value theory provides an efficient
mathematical framework to deal with these problems in the situation
where the largest values of the variables of interest tend to occur
simultaneously. This situation is referred to as {\em asymptotic
  dependence} in extreme value theory. In this case, probabilities of
simultaneous large values of the components of the vector can be
approximated and estimated by means of the multivariate extreme value
distributions.  In the opposite case of {\em asymptotic independence},
the approximate probability for two or more components being
simultaneously large given by the standard theory is zero.
Refinements of the standard theory are thus needed.

One refinement is the concept of hidden regular variation introduced by
\cite{resnick:2002}. Another approach is studying,
if it exists, the limiting distribution of a random vector
conditionally on one component being large. See  \cite{heffernan:resnick:2007}. Formally stated in the
bivariate case, this corresponds to assuming that there exist
functions $m$, $a$ and $\psi$, and a bivariate distribution function
(cdf) $K$ on $[0,\infty)\times (-\infty,\infty)$ with non degenerate margins
 such that
\begin{gather}
  \lim_{t\to\infty} \pr(X \leq t + \psi(t) x \; ; Y \leq m(t) + a(t) y
  \mid X > t) = K(x,y) \; , \label{eq:loi-limite}
\end{gather}
at all points of continuity of $K$.
\cite{das:resnick:2008} introduced the terminology of
conditional extreme-value (CEV) model. Note that Condition (\ref{eq:loi-limite}) implies that $X$ belongs to a max-domain of attraction.
More properties can be found in Section \ref{sec:CEV-EV}, where in particular the relationship between CEV models and usual multivariate extreme value (EV) models is  explicited. 
Statistical applications of the conditional model on various domains
are discussed in several papers: see \cite{heffernan:tawn:2004} for a
study on air quality monitoring,
\cite{abdous:fougeres:ghoudi:soulier:2008} and
\cite{fougeres:soulier:2008} for an insight into financial contagion,
and \cite{das:resnick:2009} for an application on Internet traffic
modeling.

An important issue that must be addressed is to study models under
which Condition~(\ref{eq:loi-limite}) holds.  The aim of this
contribution is to review existing models and exhibit new ones
satisfying~(\ref{eq:loi-limite}). We restrict our attention to bivariate random vectors for
simplicity of exposition. We focus on the case where the conditioning
variable belongs to the domain of attraction of the Gumbel
distribution. 
The reason for that is that, as mentioned later, this situation is, in the models we
consider, strongly related to the asymptotic independence, which is
precisely the case where there is an advantage to work with CEV models
instead of EV models (cf.  Section \ref{sec:CEV-EV}). Our work is an attempt to motivate
the use of CEV models by exhibiting a class of bivariate models that satisfy Condition (\ref{eq:loi-limite}).

In Section~\ref{sec:modeles}, we review the existing literature. 
In Section~\ref{sec:radial} we study new models for bivariate vectors
$(X,Y)$ with a polar representation $R(u(T),v(T))$ where $R$ is a
nonnegative random variable in the domain of attraction of the Gumbel
law, independent of the random variable $T$, and the functions $u$ and
$v$ are a parametrization of a certain curve.  This model includes
many models already studied in the literature, and in particular the
bivariate elliptical distributions. Our main result
(Theorem~\ref{theo:r-t} in Section~\ref{sec:radial}) shows that the
local geometric nature of this curve around the maximum of $u$
determines the existence and form of the limiting distribution
in~(\ref{eq:loi-limite}).  
Our result covers  situations that are more general than the results of  \cite{balkema:embrechts:2007}.
In particular, as a consequence of their local
nature, our assumptions do not imply that the conditioned variable
($Y$) belongs to the domain of attraction of a univariate extreme
value distribution. Thus these polar distributions may not be imbedded
in a standard multivariate extreme value model. But when they are, we
show that they are asymptotically independent. In order to prove this,
we extend \cite[Proposition~4.1]{das:resnick:2008}. Finally, following
\cite{abdous:fougeres:ghoudi:soulier:2008} we also study a second
order correction for the asymptotic
approximation~(\ref{eq:loi-limite}). Section~\ref{sec:preuve} contains the proof of 
Theorem~\ref{theo:r-t}. Some additional auxiliary results are given and proved in Section~\ref{sec:lemmes}.

\section{Elliptical and asymptotically elliptical distributions}
\label{sec:modeles}

Early results providing families of distributions that satisfy (\ref{eq:loi-limite})
 were obtained by
\cite{eddy:gale:1981} for spherical distributions and by
\cite{berman:1983} for bivariate elliptical distributions.  Multivariate
elliptical distributions and related distributions were investigated by
\cite{hashorva:2006,hashorva:kotz:kume:2007}.  One essential feature
of elliptical distributions is that the level sets of their density
are ellipses in the bivariate case, or ellipsoids in general. Such
geometric considerations have been deeply investigated and generalized
in many directions by \cite{barbe:2003} and
\cite{balkema:embrechts:2007}. It must be noted that these geometric
properties will be ruined by transformation of the marginal
distributions to prescribed ones. 
Another specific feature of the elliptical and related models is
that the property of asymptotic dependence or independence is related
to the nature of the marginal distributions.  If they are regularly
varying, then the components are asymptotically dependent; if the
marginal distributions belong to the maximum domain of attraction of
the Gumbel distribution, then the components are asymptotically
independent.

We start by recalling some definitions that will be used throughout
the paper. A nondecreasing function $g$ is said to {belong to the
  class $\Gamma$ or to be $\Gamma$-varying} \cite[Definition
0.47]{resnick:1987}, if there exists a positive function $\psi$ such
that
\begin{gather*}
  \lim_{x\to x_1} \frac{g(x+\psi(x) t)}{g(x)} = \mathrm{e}^{t} \; .
\end{gather*}
It is well known (cf. \citet[Theorem 1.2.5]{dehaan:ferreira:2006}) that
a random variable $X$, with cdf $F$ and upper limit of the support $x_1$, is in the max-domain of attraction
of the Gumbel distribution if and only if $1/(1-F)$ {is
  $\Gamma$-varying}, i.e.
\begin{gather}
  \lim_{x\to x_1} \frac{1-F(x+\psi(x) t)}{1-F(x)} = \mathrm{e}^{-t} \; .
  \label{eq:rapvar}  
\end{gather}
The function $\psi$ is called an {\em auxiliary function}. It is
defined up to asymptotic equivalence and necessarily satisfies
$\psi(x) = o(x)$ if $x_1 = \infty$ and $\psi(x) = o(x_1-x)$ if
$x_1<\infty$.  In the sequel, for notational simplicity, we only
consider the case $x_1 = \infty$. The modifications to be made in the
case $x_1<\infty$ are straightforward, and the main change is that the
rates of convergence must be expressed in terms of $\psi(x)$ instead
of $\psi(x)/x$.  The function $\psi$ is self-neglecting (or Beurling slowly varying, cf.
\citet[Section~2.11]{bingham:goldie:teugels:1989}), i.e.  for all
$t>0$,
\begin{gather*}
  \lim_{x\to\infty} \frac{\psi(x+\psi(x)t)}{\psi(x)} = 1 \; .
\end{gather*}

Two random variables $X$ and $Y$ such that $(X,Y)$ belongs to the
maximum domain of attraction of a bivariate extreme value distribution
$G$ are said to be {\em asymptotically independent} if $G$ has
independent marginals.  (See e.g.
\cite[Section~6.2]{dehaan:ferreira:2006}).

In the following two subsections,  we recall the results on elliptic
distributions and we state a bivariate version of a general result of
\cite{balkema:embrechts:2007} which provides  geometric sufficient
conditions for~(\ref{eq:loi-limite}) to hold.  We also point out and
illustrate the local nature of the sufficient condition formulated by
\cite{balkema:embrechts:2007}, in the case of $\Gamma$-varying
upper tails.

\subsection{Elliptical distributions}
Consider a bivariate  elliptical random vector, i.e. a
 random vector $(X,Y)$ that can be expressed as
\begin{gather} \label{eq:elliptic}
  (X,Y) = R(\cos \Theta, \rho \cos\Theta + \sigma \sin \Theta)
\end{gather}
with $\sigma^2 = 1 -\rho^2$, in terms of a positive random variable
$R$ called ``radial component'' and an ``angular'' random variable
$\Theta$ uniformly distributed on $[0, 2\pi)$. The following result
was originally proved in the case $\rho=0$ as a technical lemma under
restrictive conditions in \cite{eddy:gale:1981}.  { The general
  result was first proved in \cite{berman:1983} in the bivariate case
  (see also \cite{berman:1992} and \cite{abdous:fougeres:ghoudi:2005})} and 
\cite{hashorva:2006} in a multivariate setting. Throughout the paper
$\Phi$ will denote the cdf of the standard normal distribution.
\begin{theo}[\cite{berman:1983}]
  \label{theo:afg}
  Let $(X,Y)$ be an elliptical random vector as defined in
  (\ref{eq:elliptic}).  If the radial component $R$ is in the domain
  of attraction of the Gumbel law, i.e.  its survival function {
    $1-H$ satisfies~(\ref{eq:rapvar})} with auxiliary function $\psi$,
  then $X$ and $Y$ are also in the domain of attraction of the Gumbel
  law and
  \begin{gather} \label{eq:ellipticresult}
    \lim_{t\to\infty} \pr(X \leq t + \psi(t) x \; , Y \leq \rho t +
    \sqrt{t \psi(t)} y \mid X>t) = (1-\mathrm{e}^{-x}) \Phi(y) \; .
  \end{gather}

\end{theo}

\paragraph{Comments on Theorem 1}
The fact that $X$ has $\Gamma$-varying upper tails follows
from~(\ref{eq:ellipticresult}) by taking $y=\infty$. Since the same
result holds with reversed roles for $X$ and $Y$, the consistency
result \cite[Theorem~2.2]{das:resnick:2008} implies that $(X,Y)$
belongs to the domain of attraction of a bivariate extreme value
distribution. Moreover, since the limiting distribution
in~(\ref{eq:ellipticresult}) has two independent marginals, $X$ and
$Y$ are asymptotically independent, i.e. the limiting extreme value
distribution is the product of its marginal distributions. See also
\citet[Section~3.2]{hashorva:2005} for a proof of this property in a
related context. In the case where $R$ has a regularly varying tail, the limiting distribution
is not a product and the vector $(X,Y)$ is asymptotically dependent. See
e.g. \cite[Theorem 1, part (i)]{abdous:fougeres:ghoudi:2005}. As mentioned in the introduction, we do
not develop this case.

If the radial component $R$ has a density $h$, then the vector $(X,Y)$
has the density $f$ defined by
\begin{gather*}
  f(x,y) = \frac{h(\sqrt{x^2+(y-\rho
      x)^2/\sigma^2})}{\sqrt{x^2+(y-\rho x)^2/\sigma^2}} \; .
\end{gather*}
The level lines of the density are homothetic ellipses $x^2+(y-\rho
x)^2/\sigma^2 = c^2$.

\begin{figure}[h]
  \centering
  \includegraphics[height=7cm]{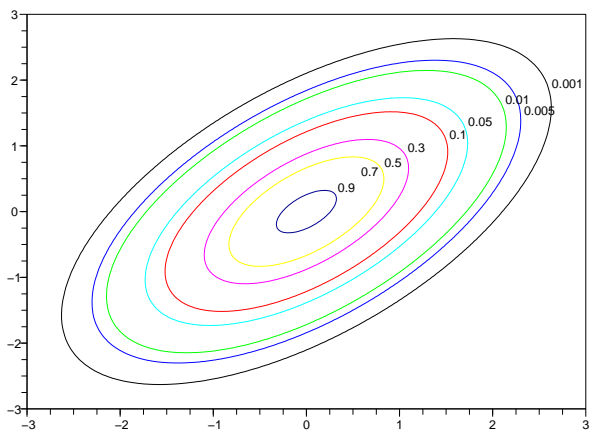}
  \caption{Level lines of the density of an elliptical distribution.
    The slope of the straight line is $\rho = .6$.}
  \label{fig:ellipses}
\end{figure}

This result can be generalized in two directions: either by weakening
the assumptions on the level lines of the density or by extending the
representation~(\ref{eq:elliptic}). The first generalization will be considered in the following
 subsection, and the second one in Section \ref{sec:radial}.

\subsection{Asymptotically elliptical distributions}
\label{sec:densite}

In this section, we state a bivariate version of \citet[Theorem
11.2]{balkema:embrechts:2007}. We first need the following definition
taken from \citet[Section~11.2]{balkema:embrechts:2007}. 
\begin{defi}
  \label{defi:plate}
  A function $L:\mathbb{R}^2\to\mathbb{R}_+$  belongs to the class
  $\mathcal L$ if for all $(x,y)\in\mathbb{R}^2$,
  \begin{gather*}
    \lim_{||(\xi,\zeta)||\to\infty} \frac{L(x+\xi,y+\zeta)}{L(\xi,\zeta)}
    = 1 \; ,
  \end{gather*}
  for any norm $||\cdot||$ on $\mathbb R^2$.
\end{defi}

\begin{hypo} \label{hypo:p-rond-n} The random vector $(X,Y)$ has a
  density $f$ such that
\begin{gather}
  f(x,y) = \mathrm{e}^{-I(x,y)} L(x,y) \; ,  \label{eq:densite-I-L}
\end{gather}
where $L \in \mathcal L$, and the function $I$  satisfies:
  \begin{gather}
    I(x,y) = p \circ n(x,y) \; , \nonumber \\
    p(r) = \int_0^r \frac{\mathrm{d} s} {\psi(s)} \; , \label{eq:p}
\end{gather}
$\psi$ is absolutely continuous with $\lim_{x\to\infty} \psi'(x) = 0$
and $n:\mathbb{R}^2\to\mathbb{R}$ is 1-homogeneous, $n^2$ is twice
differentiable and the Hessian matrix of $n$ is positive definite.
\end{hypo}

\begin{theo}[\cite{balkema:embrechts:2007}]
  \label{theo:B-E}
  Under Assumption~\ref{hypo:p-rond-n}, $X$ and $Y$ {are in the
    domain of attraction of the Gumbel law and
    satisfy~(\ref{eq:rapvar})} with auxiliary function $\psi$, are
  asymptotically independent, and there exist real numbers $\rho$ and
  $\sigma$ such that
  \begin{gather}
    \lim_{t\to\infty} \pr(X \leq t + \psi(t) x \; ; Y \leq \rho t +
    \sigma \sqrt{t\psi(t)} y \mid X>t) 
    = (1-\mathrm{e}^{-x}) \Phi(y) \; .
    \label{eq:limite-gaussienne}
\end{gather}
\end{theo}

  This result shows that Assumption~\ref{hypo:p-rond-n} is a
  sufficient condition for the limit~(\ref{eq:loi-limite}) to hold,
  with $m(x) = \rho x$, $a(x) = \sigma \sqrt{x\psi(x)}$ and $K(x,y) =
  (1-\mathrm{e}^{-x})\Phi(y)$.  The constants $\rho$ and $\sigma$ are
  characterized by the second order expansion of the function $n$
  \begin{gather*}
    n(1+x,\rho + \sigma y) = 1+x+y^2/2 +o(x^2+y^2) \; .
  \end{gather*}
  This condition implies that the tangent at the point $(1,\rho)$ to
  the curve $n(\xi,\zeta)=1$ is vertical. Note that the level lines of
  the function $n$ are not those of the density $f$ defined
  in~(\ref{eq:densite-I-L}), unless the function $L$ is constant, but,
  loosely speaking, the level lines of $f$ converge to those of $n$.
  Theorem~\ref{theo:B-E} implies Theorem~\ref{theo:afg} when the
  radial distribution is absolutely continuous.

\begin{xmpl}
  Let $h$ and $g$ be  density functions defined
  on $[0,\infty)$ and $[-\pi/2,\pi/2]$, respectively.  The function
  $f$ defined by
  \begin{gather}
    f(x,y) = \frac{h(x^2+(y-\rho x)^2/(1-\rho^2))}{\sqrt{x^2+(y-\rho
        x)^2/(1-\rho^2)}} \; g\circ\arctan((y-\rho
    x)/x\sqrt{1-\rho^2}) \label{eq:densite-exemple}
  \end{gather}
  is then a bivariate density function on $\mathbb{R}^2$.  If $g$ is a
  constant, then $f$ is the density of an elliptical vector. If $h$
  can be expressed as in~(\ref{eq:p}) and if $g$ is continuous and
  bounded above and away from zero, then $f$ satisfies
  Assumption~\ref{hypo:p-rond-n}.  Figure~\ref{fig:ellipse-flat} shows
  the level lines of such a density, with $\rho=.6$, $h(t) =
  \exp(t^2/2)/\sqrt{2\pi}$ and $g(t) = c \{1+[t^2 -
  (\pi/4)^2]^2\}$.  The level lines seem to be asymptotically
  homothetic.

\end{xmpl}

\begin{figure}[h]
  \centering
  \includegraphics[height=8cm]{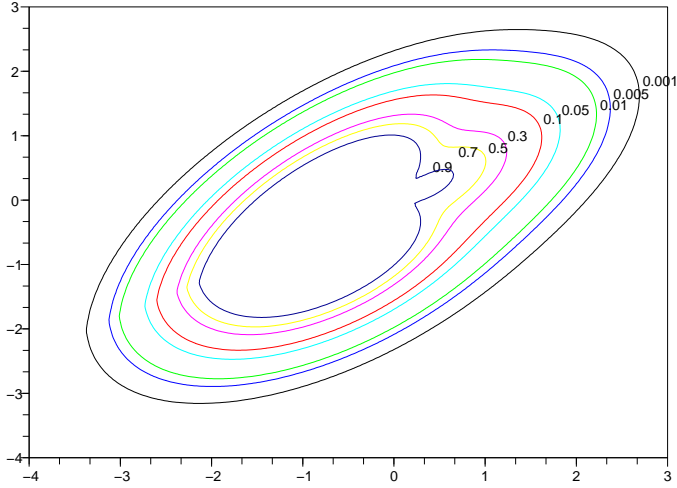}
  \caption{Level lines of the density given
    by~(\ref{eq:densite-exemple}).}
  \label{fig:ellipse-flat}
\end{figure}

%\begin{rem}
  It is important to note that under Assumption~\ref{hypo:p-rond-n}
  the normalizing functions $m$ and $a$ satisfy $a(x) = o(m(x))$,
  since in the present context $m(x) = \rho x$ and $a(x) = \sigma
  \sqrt{x\psi(x)}$ with $\psi(x) = o(x)$. This implies that only the
  local behaviour of the curve $n(\xi,\zeta)= 1$ around the point
  $(1,\rho)$ matters.  In other words, the
  limit~(\ref{eq:limite-gaussienne}) still holds if $(X,Y)$ is
  conditioned to remain in the cone $\{(\rho-\epsilon)x \leq y \leq
  (\rho+\epsilon) x\}$ for any arbitrarily small $\epsilon>0$. This
  suggests that Assumption~\ref{hypo:p-rond-n} must only be checked
  {\em locally} to obtain the limit~(\ref{eq:limite-gaussienne}).

%\label{rem:local}
%\end{rem}

\begin{xmpl}[Mixture of two bivariate Gaussian vectors]
  \label{xmpl:melange-gaussiennes}
  Let $B$ be a Bernoulli random variable such that $\pr(B = 1) = p \in
  (0,1)$. Let $X$ and $Z$ be two i.i.d standard gaussian random
  variables, $\rho \ne \tau \in[-1,1]$ and define $Y$ by
\begin{gather}\label{eq:melange-gaussiennes}
  Y = B (\rho X + \sqrt{1-\rho^2} Z) + (1-B) (\tau X +
  \sqrt{1-\tau^2} Z) \; .
\end{gather}
Then $Y$ is a standard Gaussian variable, and $(X,Y)$ is a mixture of
two Gaussian vectors. Figure~\ref{fig:melange-gauss} shows the level
curves of the density function of the pair $(X,Y)$ with $p=.4$,
$\rho=.8$ and $\tau=-.4$.  The density function of $(X,Y)$ does not
satisfy Assumption~\ref{hypo:p-rond-n}, and Theorem~\ref{theo:B-E} cannot
be applied. Indeed, applying Theorem~\ref{theo:afg} to each component
of the mixture yields
\begin{align*}
  \pr(Y & - \rho x  \leq \sqrt{1-\rho^2} z \mid X>x) \\
  & = p \, \pr(\rho
  (X-x) + \sqrt{1-\rho^2}Z \leq \sqrt{1-\rho^2}z \mid X>x) \\
  & + (1-p) \pr(\tau (X-x) + (\tau - \rho) x + \sqrt{1-\tau^2} Z \leq
  \sqrt{1-\rho^2}z \mid X>x) \\
  & \sim p\Phi(z) + (1-p) \mathbbm{1}_{\{\tau<\rho\}} \; .
\end{align*}
Thus the limiting distribution is degenerate, with a positive mass
either at~$-\infty$ or $+\infty$.
\begin{figure}[h]
  \centering
  \includegraphics[width=10cm]{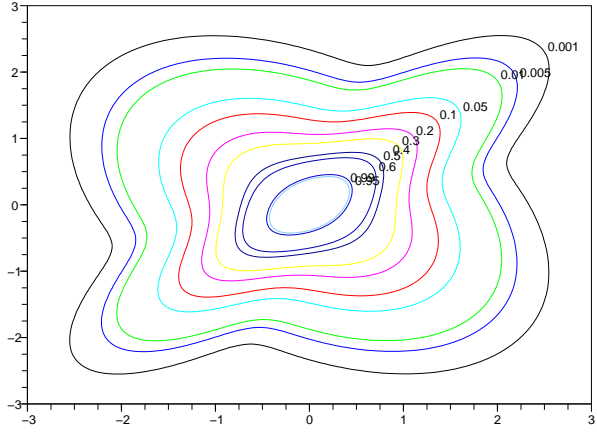}
  \caption{Level lines of the density function of the pair $(X,Y)$
    defined in Example~\ref{xmpl:melange-gaussiennes}.}
  \label{fig:melange-gauss}
\end{figure}
However, a proper limiting distribution can be obtained for $(X,Y)$
conditioned to remain in $\mathcal{C} = \{(x,y)\in\mathbb{R}^2 \mid c_1 x
\leq y \leq c_2 x\}$ for $c_1< c_2$ such that $\rho \in [c_1,c_2]$ and
$\tau \notin [c_1,c_2]$. Denote $\pr( \cdot \mid (X,Y) \in
\mathcal{C})$ by $\pr_{\mathcal C}(\cdot)$. Then
\begin{gather} \label{eq:limit-mixture}
  \lim_{x\to\infty} \pr_{\mathcal C}( Y \leq \rho x + \sqrt{1-\rho^2} z
  \mid X>x) = \Phi(z) \; .
\end{gather}
To prove this claim, we assume without loss of generality that
$\rho=0$.  Then
\begin{gather*}
  \pr_{\mathcal C}( Y \leq z \mid X>x) = \frac{ \pr(Y \leq z \; ; \
    (X,Y) \in \mathcal C \mid X>x)}{ \pr((X,Y) \in \mathcal C \mid
    X>x) } \; .
\end{gather*}
For fixed $z$ and $x>z/c_2$, it holds that
\begin{align*}
  \pr(Y & \leq z \; ; \ (X,Y) \in \mathcal C \mid X>x)  \\
  & = p \, \pr(c_1 X \leq Z \leq z \mid X>x) \\
  & \ \ \ + (1-p) \pr(c_1 X \leq \tau X + \sqrt{1-\tau^2} Z \leq z
  \mid X>x) \sim p \Phi(z) \; .
\end{align*}
Since $\rho\in[c_1,c_2]$ and $\tau\notin[c_1,c_2]$, and since $z \leq
c_2 x$, it is easily obtained that
\begin{gather*}
  \lim_{x\to\infty}   \pr(c_1 X \leq Z \leq z \mid X>x) = \Phi(z) \; ,  \\
  \lim_{x\to\infty} \pr(c_1 X \leq \tau X + \sqrt{1-\tau^2} Z \leq z
  \mid X>x) = 0 \; .
\end{gather*}
Thus $\lim_{x\to\infty} \pr(Y \leq z \; ; \ (X,Y) \in \mathcal C \mid
X>x) = p\Phi(z)$,
which proves~(\ref{eq:limit-mixture}).
\end{xmpl}

\section{Bivariate vectors with polar representation}
\label{sec:radial}

In this section, we show that Theorem~\ref{theo:afg} can be extended
from elliptical distributions to more general bivariate distributions
that admit a radial representation $R(u(T),v(T))$ where $R$ and $T$
are independent, $T$ is not uniformly distributed and $u$ and $v$ are
more general functions than in the elliptical case.  We start by
collecting the assumptions that will be needed.

\begin{hypo} \ \\ \vspace*{-2em} \label{hypo:u-v}
  \begin{enumerate}[A]
  \item \label{item:hypo-u} The function $u:[0,1]\to [0,1]$ is
    continuous, has a unique maximum 1 at a point $t_0 \in (0,1)$ and
    has an expansion
    \begin{gather}
      u(t_0+t) = 1 - \ell(t)
      \label{eq:u-local}
  \end{gather}
  where $\ell$ is decreasing from $[-\epsilon,0]$ to $[0,\eta_-]$ and increasing from
  $[0,\epsilon]$ to $[0,\eta_+]$ for some $\epsilon, \eta_-, \eta_+ >0$, and regularly varying at zero
  with index $\kappa>0$.  The functions  $\ell^\leftarrow_- : [0,\eta_-] \to  [-\epsilon,0]$ 
  and  $\ell^\leftarrow_+ : [0,\eta_+] \to  [0,\epsilon]$ respectively defined as  $\ell^\leftarrow_-(s) =\sup\{ t \in  [-\epsilon,0] : l(t)\leq s \}$
and $\ell^\leftarrow_+(s) =\inf\{ t \in  [0,\epsilon] : l(t)\geq s \}$ are absolutely
  continuous and their derivatives  $(\ell^\leftarrow_-)'(s)$ and  $(\ell^\leftarrow_+)'(s)$
  are  regularly varying at zero with index $1/\kappa-1$.
\item \label{item:hypo-v} The function $v$ defined on $[0,1]$ is
  strictly increasing in a neighborhood of $t_0$, $v(t_0) = \rho$, and
  the function $t\mapsto v(t_0+t)-\rho$ is regularly varying with
  index $\delta>0$. Its inverse $v^\leftarrow$ is absolutely
  continuous and its derivative is regularly varying at zero with
  index $1/\delta-1$.
\end{enumerate}
\end{hypo}

\begin{hypo}
  \label{hypo:T}
  The density function $g:[0,1]\to\mathbb{R}^+$ is regularly varying
  at $t_0$ with index $\tau>-1$ and bounded on the compact subsets of
  $[0,1]\setminus\{t_0\}$.
\end{hypo}

\begin{theo}
  \label{theo:r-t}
  Let $R$ {be in the domain of attraction of the Gumbel law with
    auxiliary function $\psi$, i.e.  its distribution function $H$
    satisfies~(\ref{eq:rapvar})}. Let $T$ be a random variable that
  admits a density $g$ that satisfies Assumption~\ref{hypo:T}.  Let
  the functions $u$ and $v$ satisfy Assumption~\ref{hypo:u-v} with
  $\delta<\kappa$.  Define $(X,Y) = R(u(T),v(T))$. Then,
  \begin{enumerate}[(i)]
  \item {the random variable $X$ is in the domain of attraction of the
    Gumbel law} and there exists a function $k$ regularly varying at
    zero with index $(1+\tau)/\kappa$ such that
  \begin{gather}
    \pr(X>x) \sim k(\psi(x)/x) \bar H(x) \; ; \label{eq:uppertail-x}
  \end{gather}
\item there exists a function $h$ regularly varying at zero with index
  $\delta/\kappa$ such that for all $y \in \mathbb{R}$,
  \begin{multline}
    \lim_{t\to\infty} \pr(X \leq t + \psi(t) x \; , Y \leq \rho t +
    t h(\psi(t)/t) y \mid X>t) \\
    = (1-\mathrm{e}^{-x}) H_{\kappa/\delta,(1+\tau)/\delta}(y) \; ,
    \label{eq:deltakappatau}
  \end{multline}
  with
\begin{gather*}
  H_{\eta,\zeta}(y) = \frac{ \int_{-\infty}^y
    \mathrm{e}^{-|s|^{\eta}/\eta} \, |s|^{\zeta- 1} \, \mathrm{d} s } {
    \int_{-\infty}^\infty \mathrm{e}^{-|s|^{\eta}/\eta} \, |s|^{\zeta-
      1} \, \mathrm{d} s } \; .
\end{gather*}
\end{enumerate}
\end{theo}

The proof of this result is in Section~\ref{sec:preuve}. One of its
main ingredient is the fact that the tail of $R$ is {
  $\Gamma$-varying}. This implies that the normalizing function $a(x)
= xh(\psi(x)/x)$ is $o(x)$. As a consequence, only the local behavior
of $u$ and $v$ around $(1,\rho)$ matters. This is similar to what was
observed under Assumption~\ref{hypo:p-rond-n}.

\subsubsection*{Comments on Theorem~\ref{theo:r-t}}
\begin{enumerate}[(i)]
\item The case $u(t)= \cos(t)$ and $v(t)=\sin(t)$ was considered by
  \cite{hashorva:2008polar}. The difference between the present
  results and this reference is not only that we consider more general
  functions $u$ and $v$, but more importantly that we point out the
  local nature of the assumptions on $u$ and $v$.

\item  Theorem~\ref{theo:r-t} handles situations where the assumptions of
Theorem~\ref{theo:B-E} do not hold. In some cases, the limiting
distribution is nevertheless the Gaussian distribution and the
normalization is the same as in Theorem~\ref{theo:r-t}; see
Example~\ref{xmpl:pareil}.
In other cases, the limiting distribution and the normalization differ
from those that appear in Theorem~\ref{theo:B-E}. There are two  reasons for this: the density of $T$ can vanish or be unbounded at
zero, or the curvature of the line parameterized by the functions $u$
and $v$ at the point $(u(t_0),v(t_0))$ can be infinite or zero. This
is illustrated in Example~\ref{xmpl:p-elliptique}.

The asymptotic distribution $H_{\kappa/\delta,(1+\tau)/\delta}$ is of
the so-called Weibull type with shape parameter $\kappa/\delta$. This
ratio characterizes the geometric nature of the curve
$t\mapsto(u(t),v(t))$ around $t_0$ and is independent of the
particular choice of the parametrization $t\mapsto(u(t),v(t))$. Its
right tail is lighter than the exponential distribution.  The
behavior of the density of $T$ has influence only on the less
important parameter $\zeta$.  The normalizing function $h$ also
depends only on $u$ and $v$.

\item  Denote $a(x)$ the normalizing function in~(\ref{eq:deltakappatau}):
$a(x)=xh(\psi(x)/x)$.  The assumption $\delta<\kappa$ implies that
$\psi(x)=o(a(x))$, so that $(X - x)/a(x)$ converges weakly to zero
given that $X>x$ as $x\to\infty$.  Hence
\begin{gather*}
  \lim_{x\to\infty} \pr \left( \frac{Y - \rho X}{a(x)} \leq y \mid X >
    x \right) = H_{\kappa/\delta,(1+\tau)/\delta}(y) \; .
\end{gather*}
This implies in particular that the case $\rho\ne0$ can be deduced
from the case $\rho=0$ by a linear transformation.

Since the function $h$ is regularly varying with index
$\delta/\kappa>0$ and since $\psi$ is self-neglecting, it also holds
that $a(X)/a(x)$ converges weakly to 1 given that $X>x$ as $x\to\infty$,
thus the conditional convergence also holds with random normalization:
\begin{gather*}
  \lim_{x\to\infty} \pr \left( \frac{Y - \rho X}{a(X)} \leq y \mid X >
    x \right) = H_{\kappa/\delta,(1+\tau)/\delta}(y) \; .
\end{gather*}
This is a particular case of
\citet[Proposition~5]{heffernan:resnick:2007}.

\item Note that Theorem~\ref{theo:r-t}
  states that $X$ is in the domain of attraction of the Gumbel law,
  with the same auxiliary function $\psi$ as $R$. However, little can
  be said about the tails of $Y$ other than $\pr(Y>y)\leq
  \pr(R>y/v^*)$, where $v^*=\max_{t\in[0,1]}v(t)$.  This is because
  nothing is assumed about of the behaviour of $v$ and $g$ around the
  point where $v$ has a maximum.  This would be needed to obtain an asymptotic form
for the tail of $Y$ similar to~(\ref{eq:uppertail-x}).
  It is in contrast with the situation of Theorems~\ref{theo:afg}
  and~\ref{theo:B-E}.  Nevertheless, it can first be proved that
\begin{gather} \label{eq:equiv-b}
  b_Y(t) \sim v^* b(t)
\end{gather}
where $b$ and $b_Y$ are the inverse functions of $1/\pr(R>\cdot)$ and
$1/\pr(Y>\cdot)$, respectively (see a proof in
Section~\ref{sec:lemmes}).  This implies that if $Y$ does belong to
the domain of attraction of an extreme value distribution, then this
distribution is necessarily the Gumbel law (since $Y$ has a lighter
tail than $v^*R$ and unbounded support) and $(X,Y)$ belongs to the
domain of attraction of a bivariate extreme value distribution with
independent marginals, i.e.  $X$ and $Y$ are asymptotically
independent. This is shown in Corollary~\ref{coro:CEV-AI} below.

\end{enumerate}

\subsection{Relations with CEV  and EV models}
\label{sec:CEV-EV}

As mentioned in the introduction, \cite{das:resnick:2008} referred as CEV models
the  families of distributions satisfying Condition (\ref{eq:loi-limite}). 
An important finding of  \cite{heffernan:resnick:2007} is that in such a model, it is not possible to
transform  $X$ and $Y$ to prescribed marginals $F_1(X)$ and $F_2(Y)$, for given univariate cdf $F_1, F_2$, when the
limiting distribution $K$  is the product of its marginals. Theorem~\ref{theo:r-t}  provides random vectors $(X,Y)$ for which 
the limiting distribution in~(\ref{eq:loi-limite}) is precisely a product, so that nonlinear
transformations of the marginals to prescribed marginals are
impossible. This is in
contrast with the usual practice of standard multivariate extreme
value theory. 

 As noted by
\citet[Section~1.2]{das:resnick:2009}, another advantage of the CEV
approach is that it does not require the assumption that all
components of the vector belong to the domain of attraction of a
univariate extreme value distribution. 
The relationship between CEV models and EV models has been investigated in \cite{das:resnick:2008}. They
proved in particular that if $(X,Y)$ belongs to the domain of
attraction of a bivariate extreme value distribution with asymptotic
dependence, then the CEV model does not provide any more information
than the EV model.  See also
\citet[Section~1.2]{das:resnick:2009}.  Conversely,
\citet[Proposition~4.1]{das:resnick:2009} gives  conditions under which,
if $(X,Y)$ satisfies~(\ref{eq:loi-limite}) and $Y$ belongs to the
domain of attraction of a univariate extreme value distribution, then
$(X,Y)$ belongs to the domain of attraction of a bivariate extreme
value.

The next result elucidates the relationship between the conditional
limit of Theorem~\ref{theo:r-t} and extreme value theory. It is
similar to \citet[Proposition~4.1]{das:resnick:2008} but covers cases
ruled out by this reference, as shown afterwards in
Corollary~\ref{coro:CEV-AI}. In the following, the extreme value
  distribution with index $\gamma$ is denoted by $G_\gamma$, that is to say 
  $G_\gamma(y)=\exp\{-(1+\gamma y)^{-1/\gamma}\}$ for each $y$ such that $1+\gamma y >0$ for
  $\gamma\neq 0$, and $G_0(y)=\exp\{-e^{-x}\}$.

\begin{prop} \label{prop:CEV-AI}
  Assume that the vector $(X,Y)$ satisfies~(\ref{eq:loi-limite}) and
  that $Y$ belongs to the domain of attraction of an extreme value
  distribution $G_\gamma$ with auxiliary function $\psi_Y$, i.e. 
  \begin{align*}
    \lim_{t\to\infty} \pr(Y>t+\psi_Y(t) y \mid Y>t) = - \log G_\gamma(y) \; .
  \end{align*}
  Assume moreover that  for all $y\in\mathbb{R}$, 
 \begin{align} \label{eq:condition-AI}
   \lim_{t\to\infty} \frac{b_Y(t) - m \circ b_X(t)+\psi_Y\circ b_Y(t)
     y} {a\circ b_X(t)} = +\infty \;
 \end{align}
 where $b_X$ and $b_Y$ are  the inverse functions of  $1/\pr(X>\cdot)$,
 of $1/\pr(Y>\cdot)$, respectively.  Then, $(X,Y)$ belongs to the
 domain of attraction of a bivariate extreme value distribution with
 independent marginals, i.e.  $X$ and $Y$ are asymptotically
 independent.

\end{prop}

\begin{proof}
  To prove asymptotic independence, we must show that
\begin{align} \label{eq:AI}
  \lim_{t\to\infty} t \pr(X & > b_X(t) + \psi\circ b_X(t) x \;, \; Y >
  b_Y(t) + \psi_Y\circ b_Y(t)y ) = 0 \; ,
\end{align}
where  $\psi$ is as
in~(\ref{eq:loi-limite}). For any $y\in\mathbb{R}$, define $\tilde y$ by
$b_Y(t) + \psi_Y\circ b_Y(t)y = m\circ b_X(t) + a\circ b_X(t) \tilde
y$, i.e.
\begin{align*}
  \tilde y & = \frac{b_Y(t) - m \circ b_X(t) + \psi_Y\circ b_Y(t)y}
  {a\circ b_X(t)} \; .
\end{align*}
Thus, by~(\ref{eq:condition-AI}), for any $z\in\mathbb{R}$, we have,
\begin{align*}
  \limsup_{t\to\infty} & \; t \pr(X > b_X(t) + \psi\circ b_X(t) x \;,
  \; Y >  b_Y(t) + \psi_Y\circ b_Y(t)y ) \\
  & =  \limsup_{t\to\infty}  \; t \pr(X > b_X(t) + \psi\circ b_X(t) x \;,
  \; Y >  m\circ b_X(t) + a\circ b_X(t) \tilde y ) \\
  & \leq \limsup_{t\to\infty} t \pr(X > b_X(t) + \psi\circ b_X(t)
  x \;, \; Y > m \circ  b_X(t) + a \circ b_X(t) z ) \\
  & = \pr(X^*>x \;, \ Y^*>z) \; , 
\end{align*}
where $(X^*,Y^*)$ is a random vector with distribution $K$.  This can
be made arbitrarily small by choosing $z$ large enough, so
(\ref{eq:AI}) holds.
\end{proof}
A simple example is provided by the elliptical distributions with identical margins, for which one has:
$b_X=b_Y$, $\psi_Y(x)=o(x), \psi_Y(x)=o(a(x))$ and $m(x)=\rho x$, so that~(\ref{eq:condition-AI}) holds. A more general result is provided by the following corollary. 
\begin{coro} \label{coro:CEV-AI}
  Under the assumptions of Theorem~\ref{theo:r-t}, if moreover $Y$
  belongs to a domain of attraction, then it is the Gumbel law,
  and~(\ref{eq:condition-AI}) holds, so that  $X$ and $Y$ are
  asymptotically independent.
\end{coro}

\begin{proof}
  Under the assumptions of Theorem~\ref{theo:r-t}, the normalizing
  functions in~(\ref{eq:loi-limite}) satisfy $m(t) = \rho t$, $\psi(t)
  = o(a(t))$ and $a(t) = o(t)$. Since $R$ has unbounded support,
  then so has $Y$ and since $Y$ has lighter tails than $v^*R$, the 
  max-domain of attraction of $Y$ can only be the Gumbel law. Thus it
  also holds that $\psi_Y(t) = o(t)$.  By~(\ref{eq:equiv-b}) we know
  that $b_Y(t) \sim v^*b(t)$ and~(\ref{eq:uppertail-x}) implies that
  $b_X(t)\sim b(t)$.  Thus we have
  \begin{gather*}
    \frac{b_Y(t) - m \circ b_X(t) + \psi_Y\circ b_Y(t)y} {a\circ
      b_X(t)} \sim \frac{(v^*-\rho) b(t)}{a \circ b_X(t)}.
  \end{gather*}
The assumptions on $v$ imply that $v^*>\rho$, so 
\begin{gather*}
  \lim_{t\to\infty} \frac{(v^*-\rho) b(t)}{a \circ b_X(t)} = \infty \; .
\end{gather*}
Thus~(\ref{eq:condition-AI}) holds for all $y\in\mathbb{R}$. 
\end{proof}

\paragraph{Remark} \citet[Proposition~4.1]{das:resnick:2008}
show the same result under a condition which can be expressed with the
present notation as $\lim_{t\to\infty} \psi_Y \circ b_Y(t)/a\circ
b_X(t) \in(0,\infty]$. Under the conditions of Theorem~\ref{theo:r-t},
if moreover $v$ and $g$ satisfy some smoothness assumptions around the
maximum of $v$ similar to Assumptions~\ref{hypo:u-v}\ref{item:hypo-u}
and~\ref{hypo:T}, it can be shown that $\psi_Y \sim \psi$ and thus
$\lim_{t\to\infty} \psi_Y \circ b_Y(t)/a\circ b_X(t) =0$, so that
\citet[Proposition~4.1]{das:resnick:2008} cannot be applied here. 

\subsection{Some  applications  of Theorem~\ref{theo:r-t}}
\label{sec:examples}
  
  We now give some examples of applications of Theorem~\ref{theo:r-t}.

\begin{xmpl} \label{xmpl:pareil}
  If the density $g$ of the variable $T$ has a positive limit at
  $t_0$, then $\tau=0$. If $u$ is twice differentiable with $u'(t_0)
  =0$ and $u^{\prime\prime}(t_0)\ne0$, then $\kappa = 2$ and if
  $v'(0)>0$, then $\delta = 1$. If these three conditions hold, the
  limiting distribution is the standard Gaussian and the normalization
  is $\sigma\sqrt{x\psi(x)}$.  Figure~\ref{fig:ellipse-flat} also
  illustrates this case. As will be shown in
  Section~\ref{sec:relation}, this is actually a particular case of
  Theorem~\ref{theo:B-E}.
\end{xmpl}

\begin{xmpl}
  \label{xmpl:p-elliptique} \cite{hashorva:kotz:kume:2007} have
  introduced a generalisation of the elliptical distributions, which
  they called $L_p$-Dirichlet distributions for all $p>0$. We consider
  here the case $p>1$.  Instead of being ellipses, the level lines of
  the density of these distributions have the following equation:
  \begin{gather} \label{eq:p-ellipses}
    |x|^p + \frac{|y-\rho x|^p}{1-|\rho|^p} =  1 \; ,
  \end{gather}
  with $\rho\in(-1,1)$. See Figure~\ref{fig:mou-tordu}. To simplify
  the discussion, we consider the case $\rho=0$.  An admissible
  parametrization is given by $u(t) = (1-t^p)^{1/p}$ and $v(t) = t$,
  which yields $\delta = 1$ and $\kappa = p$. Thus,
  Assumption~\ref{hypo:p-rond-n} does not hold except if $p=2$, which
  is the elliptical context. If the density of $T$ has a positive
  limit at zero, then $\tau=0$ and the cdf of the limiting
  distribution is then
\begin{gather*}
  H_{p,1}(y) = \frac{ \int_{-\infty}^y \mathrm{e}^{-|s|^p/p} \, \mathrm{d} s } {2
    p^{1/p-1} \, \Gamma(1/p) } \; .
\end{gather*}
\end{xmpl}

%\clearpage

\begin{figure}[htbp]
  \centering \includegraphics[width=6cm]{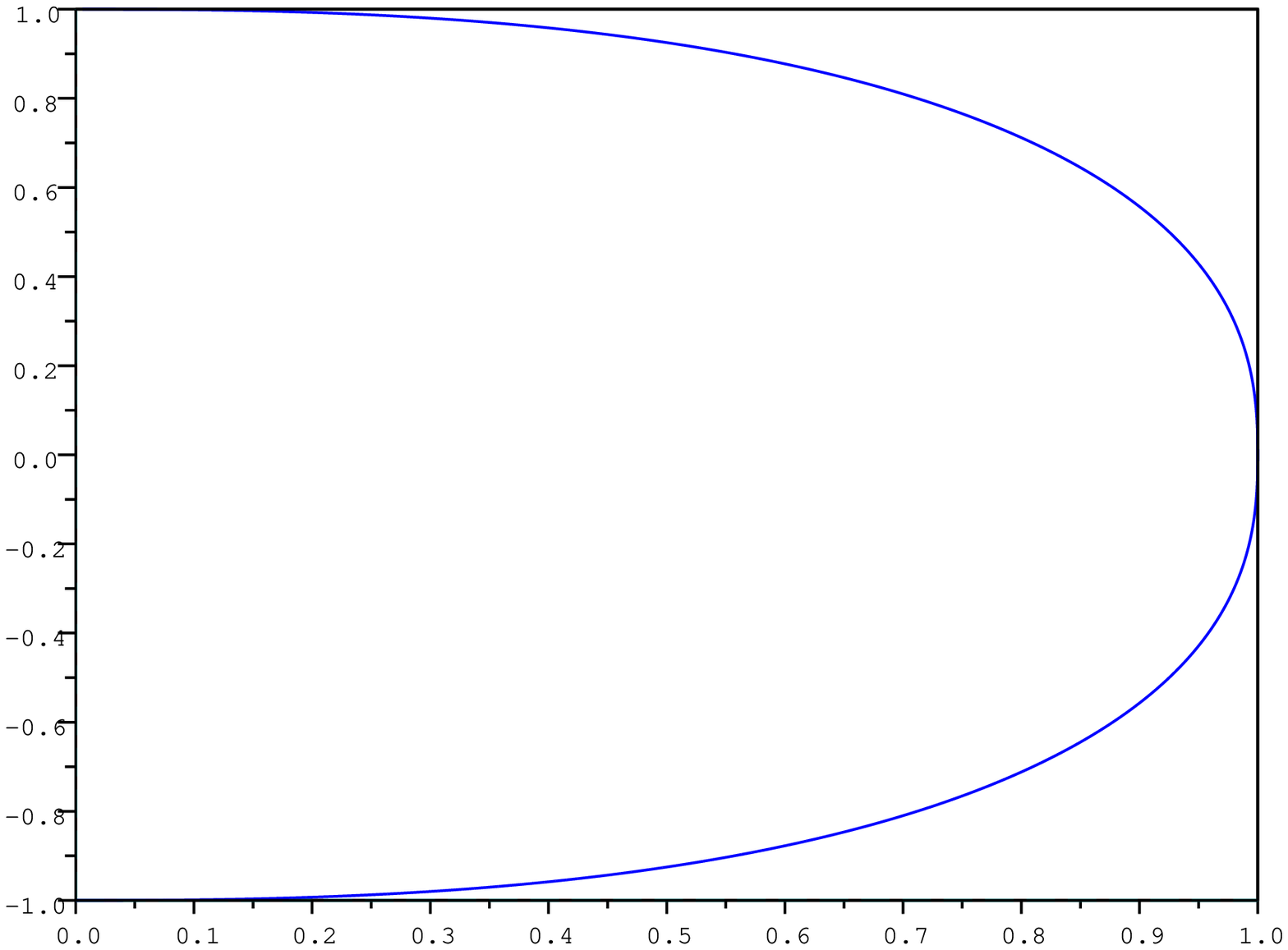}
  \includegraphics[width=6cm]{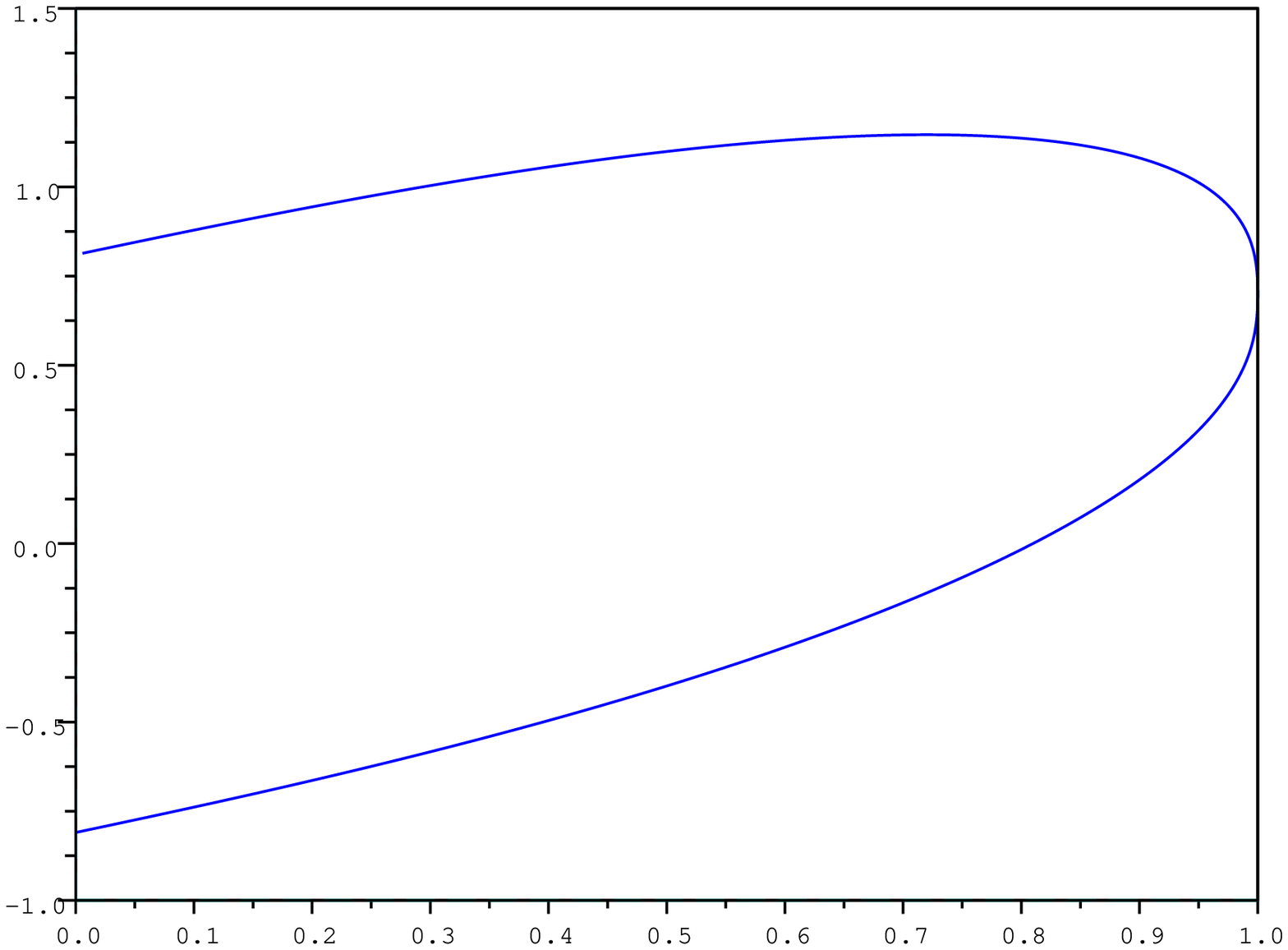}
  \caption{The curve given by Equation~(\ref{eq:p-ellipses})
    with $p=3/2$, $\rho=0$ (left) and $\rho=.7$ (right).}
  \label{fig:mou-tordu}
\end{figure}

As a last remark in this section, note that the
result~(\ref{eq:uppertail-x}) on the upper tail of $X$ is actually a
particular case of a more general result on the tail of the product of
a random variable $X$ in the domain of attraction of the Gumbel law
and a bounded random variable $U$.  Similar results exist for
heavy-tailed or subexponential distributions, for instance the
celebrated Breiman's Lemma (see \cite{breiman:1965} and
\cite{cline:samorodnitsky:1994}). In the present context,
\citet[Theorem~3]{hashorva:2008polar} states a version of this result
with $U = \cos\Theta$ and $\Theta$ has a density on $(-\pi,\pi)$. We
state it under slightly more general assumptions which highlight the
fact that only the behaviour of $U$ near its maximum must be
specified.

\begin{prop}
  Let $R$ be a nonnegative random variable whose cdf $H$ is in the
  domain of attraction of the Gumbel law.  Let $U$ be a nonnegative
  random variable, independent of $R$, such that $U \leq b < \infty$
  a.s. and that admits a density $g$ in a neighborhood of $b$ which is
  regularly varying at $b$ with index $\tau>-1$. Then $RU$ {
    satisfies~(\ref{eq:rapvar})} with auxiliary function $\psi(x/b)$
  and
\begin{align}
  \pr(RU>x) \sim b^2\Gamma(\tau+1) \frac{\psi(x/b)}x \,
  g(\{1/b+\psi(x/b)/x\}^{-1}) \bar H(x/b) \; . \label{eq:survie-RU}
\end{align}
\end{prop}

\begin{proof}
  The proof is along the lines of the proof of~(\ref{eq:uppertail-x}),
  and makes use of Lemma~\ref{lem:convunifcompact}.
\end{proof}

\subsection{Relation between Theorems~\ref{theo:B-E} and~\ref{theo:r-t}}
\label{sec:relation}

Let $X,Y$ be random variables whose joint density $f$ can be expressed
as
\begin{gather}
  f(x,y) = g \circ n(x,y) \; ,   \label{eq:g-rond-n}
\end{gather}
where $g$ is a nonnegative function on $\mathbb{R}_+$ such that
$\int_0^\infty r g(r)\, \mathrm{d} r < \infty$, $n$ is a positively
homogeneous function with index~1 and the level line $n(x,y)=1$ admits
the parametrization $t\to(u(t),v(t))$, $t\in[0,1]$.  The change of
variable $x=ru(t)$, $y=rv(t)$ yields, for any bounded measurable
function~$\varphi$:
\begin{gather*}
  \esp[\varphi(X,Y)] = \int_0^\infty \int_0^1 \varphi(ru(t),rv(t)) \, r g(r) \,
  |u(t)v'(t) - u'(t)v(t)| \, \mathrm{d} r \, \mathrm{d} t \; .
\end{gather*}
Hence $(X,Y) = R(u(T),v(T))$, where $R$ has density $c(g)^{-1} r
g(r)$, with $c(g) = \int_0^\infty r g(r) \, d r$ and $T$ has density
$c(u,v)^{-1} |uv' - u'v|$ with $c(u,v) = \int_0^1 |u'(t)v(t) - u(t)
v'(t)| \mathrm{d} t$.

Conversely, if $(X,Y) = R(u(T),v(T))$ where $R$ and $T$ are
independent, $R$~has a density $h$ on $[0,\infty)$, and if there
exists a point $t_0$ such that $u'(t_0)=0$,
$u^{\prime\prime}(t_0)>0$ and $v'(t_0)>0$, then the function $v/u$
is invertible on an interval
around $t_0$. Let $\phi$ be its inverse, and define $n(x,y) =
x/\{u\circ \phi(y/x)\}$. Without loss of generality, we can assume
that $u(t_0) = 1$ and denote $v(t_0) = \rho$. Then $n$ is a
positively homogeneous function with index~1, and if $T$ admits a
density $g$, then $(X,Y)$ has a density $f$ in a cone $C = \{(x,y)
\mid (\rho-\epsilon)x \leq y \leq (\rho+\epsilon)x\}$ around the
line $y=\rho x$, defined by
\begin{gather*}
  f(x,y) = \frac{h(n(x,y))}{n(x,y)} \; L(x,y) \; ,
\end{gather*}
with
\begin{gather*}
  L(x,y) = \frac{ g\circ \phi(y/x)}{|u'v-uv'|\circ\phi(y/x)} \; .
\end{gather*}
If the density $g$ and the Jacobian $|u'v-uv'|$ are both positive and
continuous at $t_0$, then the function $L$ belongs to the class ${\cal L}$. See
Lemma~\ref{lem:plate} for a proof. Thus Assumption~\ref{hypo:p-rond-n}
locally holds, and Theorem~\ref{theo:B-E} implies
Theorem~\ref{theo:r-t} in this context.

\subsection{Second order correction}
As illustrated in \cite{abdous:fougeres:ghoudi:soulier:2008}, it is
useful for statistical purposes to have a second order correction to
the asymptotic approximation (\ref{eq:deltakappatau}) provided by Theorem~\ref{theo:r-t}.  In
order to obtain such a refinement, rates of convergence in all the
approximations used to prove Theorem~\ref{theo:r-t} are needed. To
simplify the discussion, we will consider the following additional
assumptions.
\begin{itemize}
\item The random variable $T$ is uniformly distributed over $[0,1]$.
\item There exist $\lambda>0$ and $\sigma>0$ such that
   \begin{gather*}
    u(t_0+t) = 1 - \frac{ \sigma^2 t^2}2 + o(t^3) \; , \ \
    v(t_0+t) = \rho + \lambda  t + O(t^2) \; .
 \end{gather*}
\item There exist functions $\chi$ and $B$ such that
  \begin{gather} \label{eq:secondordre}
    \left| \frac{\bar H(x+t\psi(x))}{\bar H(x)} - \mathrm{e}^{-t} \right|
    \leq \chi(x) B(t) \; ,
  \end{gather}
  for all $t\geq 0$ and $x$ large enough, where $\lim_{x\to\infty}
  \chi(x) = 0$, and $B$ is bounded on the compact subsets of
  $[0,\infty)$ and integrable over $[0,\infty)$.
\end{itemize}
The bound~(\ref{eq:secondordre}) is a nonuniform rate of convergence.
See \citet[section 2.2]{abdous:fougeres:ghoudi:soulier:2008} for examples.
Under these assumptions, it is possible to obtain a rate of
convergence and a second order correction in Theorem~\ref{theo:r-t}.
Proceeding as in the proof of
\citet[Theorem~3]{abdous:fougeres:ghoudi:soulier:2008} yields
\begin{multline*}
  \pr( Y \leq \rho x + \frac\lambda\sigma \sqrt{x \psi(x)} z \mid X>x)
  \\
  = \Phi(z) - \frac\rho\lambda \sqrt{\frac{\psi(x)}x} \phi(z) +
  O(\chi(x)) + o\left( \sqrt{\psi(x)/x}\right) \; .
\end{multline*}
Replacing $z$ by $z+\rho\lambda^{-1}\sqrt{\psi(x)/x}$ yields
 \begin{multline*}
   \pr\left( Y \leq \rho x + \frac\lambda\sigma \sqrt{x \psi(x)} z +
     \frac\rho\sigma \psi(x)\mid X>x\right) \\
   = \Phi(z)+ O(\chi(x)) + o\left( \sqrt{\psi(x)/x}\right) \; .
\end{multline*}
This second order correction is meaningful only if $\chi(x) =
o(\sqrt{\psi(x)})$. The improvement is only moderate here: the bound is
$o(\sqrt{\psi(x)/x})$ instead of $O(\sqrt{\psi(x)/x})$, because we
assumed only that $u(t_0+t) = 1 - \sigma^2 t^2/2 + o(t^3)$. If the
expansion of $u$ around $t_0$ is $u(t_0+t) = 1 - \sigma^2 t^2/2 +
O(t^4)$, then the bound becomes $O(\psi(x)/x)$. This is the case for
bivariate elliptical distributions.

\section{Proof of Theorem~\ref{theo:r-t}}
\label{sec:preuve}
Let $\epsilon$ be defined as in Assumption~\ref{hypo:u-v}.
Since $u$ has its maximum at $t_0$, there exists
$\eta>0$ such that for all $t\notin [t_0-\epsilon,t_0+\epsilon]$, it
holds that $u(t) \leq 1-\eta$.  Then,
  \begin{align*}
    \pr(X>x \; ; \ Y>y) & = \int_0^1 \bar H\left( \frac x{u(t)} \vee
      \frac y{v(t)} \right)  g(t) \, \mathrm{d} t \\
    & = \int_{|t-t_0| \leq \epsilon} \bar H\left( \frac x{u(t)} \vee
      \frac y{v(t)} \right) g(t) \, \mathrm{d} t \\
    & + \int_{|t-t_0| > \epsilon} \bar H\left( \frac x{u(t)} \vee
      \frac y{v(t)} \right) g(t) \, \mathrm{d} t \; .
  \end{align*}
  Let $r(x)$ denote the last term. The bound~(\ref{eq:reste}) in
  Lemma~\ref{lem:borne-afg} yields that for any $p>0$,
  \begin{align*}
    r(x) \leq \bar H(x/(1-\eta)) = O\left( \{\psi(x)/x\}^p \bar H(x)
    \right) \; .
  \end{align*}
  This will prove that $r(x)$ is negligible with respect to the first
  integral for which we now give an asymptotic equivalent.  By
  assumption, we can choose $\epsilon$ such that the function $v/u$ is
  continuous and increasing on $[t_0-\epsilon,t_0+\epsilon]$, because
  $\delta<\kappa$, so that $v/u \sim v$ in a neighborhood of $t_0$.

  If $y$ can be expressed as $y=\rho x+o(x)$, then for large $x$, it
  holds that $y/x \in [(v/u)(t_0-\epsilon),(v/u)(t_0+\epsilon)]$. If
  $y/x > \rho = (v/u)(t_0)$, then there exists $t_1 \in
  [t_0,t_0+\epsilon]$ such that $(v/u)(t_1) = y/x$.  Thus,
  \begin{multline*}
    \int_{t_0-\epsilon}^{t_0+\epsilon} \bar H\left( \frac x{u(t)} \vee
      \frac y{v(t)} \right) g(t) \, \mathrm{d} t \\
    = \int_{t_0-\epsilon}^{t_1} \bar H\left( \frac y{v(t)} \right)
    g(t) \, \mathrm{d} t + \int_{t_1}^{t_0+\epsilon} \bar H\left( \frac
      x{u(t)}  \right) g(t) \, \mathrm{d} t \; .
  \end{multline*}
  Let $I$ and $J$ denote the last two integrals, respectively. The
  successive changes of variables $s=1/u(t)$ and  $s=1+\omega\psi(x)/x$ yield
  \begin{multline*}
    J = \int_{1/u(t_1)}^{1/u(t_0+\epsilon)} \bar H(xs)
    \frac{-(u^\leftarrow)'(1/s)}{s^2}
    g(u^\leftarrow(1/s)) \, \mathrm{d} s \\
    = \frac{\psi(x)}x
    \int_{x\{1/u(t_1)-1\}/\psi(x)}^{x\{1/u(t_0+\epsilon)-1\}/\psi(x)}
    \bar H(x+\psi(x)\omega) \\
    \times
    \frac{-(u^\leftarrow)'(1/\{1+\omega\psi(x)/x\})}{\{1+\omega\psi(x)/x\}^2}
    g(u^\leftarrow(1/\{1+\omega\psi(x)/x\})) \, \mathrm{d} \omega \; .
  \end{multline*}
Let $k_1$ denote the function defined by
\begin{align*}
  k_1(\omega) = -\omega (u^\leftarrow)'(1/\{1+\omega\})
  g(u^\leftarrow(1/\{1+\omega\})) \; .
\end{align*}
Assumptions~\ref{hypo:u-v} and~\ref{hypo:T} imply that $k_1$ is
regularly varying at zero with index $(1+\tau)/\kappa$ and
Lemma~\ref{lem:convunifcompact} yields
\begin{align*}
  J \sim k_1\{\psi(x)/x\} \bar H (x) \int_{z_0}^\infty \mathrm{e}^{-t}
  t^{(1+\tau)/\kappa-1} \, \mathrm{d} t \; ,
\end{align*}
where $z_0 = x\{1/u(t_1)-1\}/\psi(x)$ and if $y$ is chosen in such a
way that $z_0$ has a finite limit when $x\to\infty$.  Set
$y/x=\rho+\xi$. Then, by definition of $t_1$, $z_0 = x\{1/u\circ
(v/u)^\leftarrow(\rho+\xi)-1\}/\psi(x)$.  Assumption~\ref{hypo:u-v} implies
that the function $\xi \mapsto u\circ(v/u)^\leftarrow(\rho+\xi)$ is
regularly varying at zero with index $\kappa/\delta$. Define an
increasing function $h$ which is regularly varying at zero with index
$\delta/\kappa$ by
\begin{gather*}
  h(x) = (v/u)\circ u^\leftarrow (1/\{1+x\}) - \rho = (1+x)v\circ
  u^\leftarrow (1/\{1+x\}) - \rho \; .
\end{gather*}
For $z\geq0$, set $y/x = \rho + h(\psi(x)/x)z$.  Then
\begin{gather*}
z_0 = \frac{x}{\psi(x)} h^\leftarrow (h(\psi(x)/x)z) \sim
z^{\kappa/\delta} \;, \\
J \sim  k_1\{\psi(x)/x\} \bar H (x) \int_{z^{\kappa/\delta}}^\infty \mathrm{e}^{-t}
t^{(1+\tau)/\kappa-1} \, \mathrm{d} t \; .
\end{gather*}
We next deal with the integral $I$, still in the case $y/x>\rho$.
Noting that $\{y/v(t_1)-x\}/\psi(x) = z_0$, the changes of variables
$s=1/v(t)$ and $s=\{x+\psi(x)\omega)\}/y$ yield
  \begin{align*}
    I & = \int_{1/v(t_1)}^{1/v(t_0-\epsilon)} \bar H(ys)
    \frac{(v^\leftarrow)'(1/s)}{s^2} g(v^\leftarrow(1/s)) \, \mathrm{d} s
    \\
    & = \frac{y \psi(x)}{x^2}
    \int_{z_0}^{\{y/v(t_0-\epsilon)-x\}/\psi(x)} \bar
    H(x+\psi(x)\omega)
    \\
    & \phantom{\frac{y \psi(x)}{x^2} } \times
    \frac{(v^\leftarrow)'((y/x)/\{1+\omega\psi(x)/x\})}{\{1+\omega\psi(x)/x\}^2}
    g(v^\leftarrow((y/x)/\{1+\omega\psi(x)/x\})) \mathrm{d} \omega \; .
  \end{align*}
The choice  $y/x=\rho+h(\psi(x)/x)z$ also yields
\begin{align*}
  \frac{y/x}{1+\omega\psi(x)/x} & = \frac{\rho + h(\psi(x)/x)
    z}{1+\omega \psi(x)/x} \sim \rho + h(\psi(x)/x) z \; .
\end{align*}
Let the function $k_2$ be defined by
\begin{align*}
  k_2(\xi) = \xi (v^\leftarrow)'(\rho+h(\xi))
  g\circ v^\leftarrow(\rho+h(\xi)) \; .
\end{align*}
Assumptions~\ref{hypo:u-v} and~\ref{hypo:T} imply that $k_2$ is
regularly varying at zero with index $(1+\tau)/\kappa
+1-\delta/\kappa>(1+\tau)/\kappa$.  Lemma~\ref{lem:convunifcompact}
yields
\begin{align*}
  I \sim \rho z^{(1+\tau)/\delta-1} k_2(\psi(x)/x) \bar H(x)
  \int_{z^{\kappa/\delta}}^\infty \mathrm{e}^{-t} \, \mathrm{d} t = o(J) \; .
\end{align*}
The case $y/x<\rho$ can be dealt with similarly and is omitted. The
remaining of the proof of assertions (i) and (ii) is straightforward.
\qed

\section{Lemmas}
\label{sec:lemmes}
The following Lemma is a straightforward consequence of the
representation theorem for the class $\Gamma$
\cite[Theorem~3.10.8]{bingham:goldie:teugels:1989}.  The argument was
used in the proof of \citet[Theorem~1]{abdous:fougeres:ghoudi:2005}.
We briefly recall the main lines of the proof for the sake of
completeness.

\begin{lem}  \label{lem:borne-afg}
  Let $H$ be a cdf in the domain of attraction of the Gumbel law
  infinite right endpoint.  For any $p >0$, there exists a constant
  $C$ such that for all $x$ large enough, and all $t\geq0$,
  \begin{gather}     \label{eq:borne-afg}
    \frac{\bar H(x+\psi(x)t)}{\bar H(x)} \leq C (1 + t)^{-p} \; .
    \\
    \frac{\bar H(\alpha x)}{\bar H(x)} \leq C (\psi(x)/x)^p \; .
    \label{eq:reste}
  \end{gather}

\end{lem}
\begin{proof}
    If $\gamma = 0$, the function $\bar H$ can be expressed as
\begin{gather*}
  \bar H(x) = c(x) \exp\left\{-\int_{x_0}^x \frac{\mathrm{d} s}{\psi(x)}
  \right\}
\end{gather*}
where $\lim_{x\to\infty} c(x) = c \in(0,\infty)$ and
$\lim_{x\to\infty} \psi'(x) = 0$. Thus, for any $\epsilon>0$ and $x$
large enough, there exists a constant $C$ such that
\begin{gather*}
  \frac{c(x+\psi(x)t)}{c(x)} \leq C \; , \ \ \
  \frac{\psi(x+\psi(x)t)}{\psi(x)} \leq 1 + \epsilon t \; .
\end{gather*}
Hence
\begin{align*}
  \frac{\bar H(x+\psi(x)t)}{\bar H(x)}
  & \leq C \exp\left\{-\int_{0}^{t} \frac{\mathrm{d}
      s}{1+\epsilon s} \right\} = C (1+\epsilon t)^{-1/\epsilon} \; .
\end{align*}
This proves~(\ref{eq:borne-afg}). The bound~(\ref{eq:reste}) follows
trivially from~(\ref{eq:borne-afg}) by choosing $\epsilon<1/p$ and by
setting $t = (\alpha-1)x/\psi(x)$.
\end{proof}

\begin{lem}    \label{lem:convunifcompact}
  Let $H$ be a cdf in the domain of attraction of the Gumbel law with
  infinite right endpoint.  Let $g$ be a function regularly varying at
  zero with index $\tau>-1$ and bounded on compact subsets of
  $(0,\infty]$.  Then
  \begin{gather*}
    \lim_{x\to\infty} \int_z^\infty \frac{\bar H(x+\psi(x)t)}{\bar
      H(x)} \frac{g(t\psi(x)/x)}{g(\psi(x)/x)} \, \mathrm{d} t =
    \int_z^\infty t^\tau \, \mathrm e^{-t} \, \mathrm{d} t \; ,
  \end{gather*}
  locally uniformly with respect to $z\geq0$.
\end{lem}
\begin{proof}
  Denote $\chi(x) = \psi(x)/x$; then $\lim_{x\to\infty} \chi(x)=0$.
  By assumption, $\bar H(x+\psi(x)t)/\bar H(x)$ converges to $\mathrm
  e^{-t}$ and $g(\chi(x)t)/g(\chi(x))$ converges to $t^\tau$, and both
  convergences are uniform on compact sets of $(0,\infty)$. It is thus
  sufficient to prove that
  \begin{gather}
    \lim_{A\to\infty} \limsup_{x\to\infty} \int_A^\infty \frac{\bar
      H(x+\psi(x)t)}{\bar H(x)} \frac{g(\chi(x)t)}{g(\chi(x))} \, \mathrm{d} t
    = 0 \; , \label{eq:alinfini}
    \\
    \lim_{\eta\to0} \limsup_{x\to\infty} \int_0^\eta \frac{\bar
      H(x+\psi(x)t)}{\bar H(x)} \frac{g(\chi(x)t)}{g(\chi(x))} \, \mathrm{d} t =
    0 \; . \label{eq:enzero}
  \end{gather}
  Let $\epsilon>0$ be such that $1/\epsilon - 1 > \tau$.  By
  Lemma~\ref{lem:borne-afg}, for large enough $x$,
  \begin{align*}
    \int_A^\infty \frac{\bar H(x+\psi(x)t)}{\bar H(x)} g(\chi(x)t) \,
    \mathrm{d} t &
    \leq C \int_A^\infty t^{-1/\epsilon} g(\chi(x)t) \, \mathrm{d} t \\
    & = C \chi(x)^{1/\epsilon-1}\int_{A \chi(x)}^\infty
    t^{-1/\epsilon} g(t) \, \mathrm{d} t \; .
  \end{align*}
  Since $g$ is locally bounded on $(0,\infty]$ and
  $-1/\epsilon+\tau<-1$, Karamata's Theorem
  (cf. for instance \citet[Proposition~1.5.10]{bingham:goldie:teugels:1989}) implies that
  there exists a constant $C'$ such that
\begin{align*}
  \int_{A \chi(x)}^\infty t^{-1/\epsilon} g(t) \, \mathrm{d} t & \leq C'
  (\chi(x)A)^{1-1/\epsilon}g(\chi(x)A)\; , \\
  \limsup_{x\to\infty} \int_A^\infty \frac{\bar H(x+\psi(x)t)}{\bar
    H(x)} \frac{g(\chi(x)t)}{g(\chi(x))} \, \mathrm{d} t & \leq
  CC'A^{1-1/\epsilon}
  \limsup_{x\to\infty} \frac{g(\chi(x)A)}{g(\chi(x))} \\
  & = CC'A^{1-1/\epsilon+\tau} \to 0
\end{align*}
as $A$ tends to infinity, because $1-1/\epsilon+\tau<0$.  This
proves~(\ref{eq:alinfini}).  Since $\bar H(x+\psi(x)t)/\bar H(x) \leq
1$ and by Karamata's Theorem, we get, for some constant  $C$,
\begin{align*}
  \int_0^\eta \frac{\bar H(x+\psi(x)t)}{\bar H(x)} g(\chi(x)t) \, \mathrm{d} t
  & \leq \int_0^\eta g(\chi(x)t) \, \mathrm{d} t \\
  & = \chi(x)^{-1} \int_{0}^{\chi(x) \eta} g(t) \, \mathrm{d} t \leq C \eta
  g(\chi(x)\eta) \; .
\end{align*}
Hence
\begin{gather*}
\lim_{\eta\to0}  \limsup_{x\to\infty} \int_0^\eta \frac{\bar H(x+\psi(x)t)}{\bar
    H(x)} \frac{g(\chi(x)t)}{g(\chi(x))} \, \mathrm{d} t \leq \lim_{\eta\to0} C
  \eta^{1+\tau} = 0\; ,  
\end{gather*}
because $1+\tau>0$, which proves~(\ref{eq:enzero}).
\end{proof}

\begin{lem}
  Let $\ell$ be a continuous function defined on $[0,\infty)$, bounded above and
  away from zero and with a finite limit at infinity.  Define $L$ on
  $\mathbb{R}_+\times\mathbb{R}_+$ by $L(x,y) = \ell(y/x)$. Then $L$ belongs to
  the class $\mathcal{L}$.
\label{lem:plate}
\end{lem}
\begin{proof}
  Since $\ell$ is bounded, it suffices to prove that if the limit
  $\lim_{(\xi,\zeta)\to\infty}
  \ell((x+\xi)/(y+\zeta))/\ell(\xi/\zeta)$ exists, then it is equal
  to~1.  Since moreover $\ell$ is continuous and bounded away from
  zero, it is enough to consider subsequences and to show that if
  $\|(\xi_n,\zeta_n)\| \to\infty$ and if $\lim_{n\to\infty} \xi_n$ and
  $\lim_{n\to\infty} \zeta_n$ both exist, then $ \lim_{n\to\infty}
  (x+\xi_n)/(y+\zeta_n) = \lim_{n\to\infty} \xi_n/\zeta_n$. Three
  cases arise.
  \begin{enumerate}[(i)]
  \item If $\lim_{n\to\infty} \xi_n = \lim_{n\to\infty} \zeta_n=
    \infty$, then
   $$
   \frac{x+\xi_n}{y+\zeta_n} = \frac{\xi_n}{\zeta_n} \,
   \frac{1+x/\xi_n}{1+y/\zeta_n} \sim \xi_n/\zeta_n \; .
   $$
 \item If $\lim_{n\to\infty} \xi_n < \infty$ and $\lim_{n\to\infty}
   \zeta_n= \infty$, then
   \begin{gather*}
     \lim_{n\to\infty} (x+\xi_n)/(y+\zeta_n) = 0 = \lim_{n\to\infty}
     \xi_n/\zeta_n \; .
   \end{gather*}
 \item If $\lim_{n\to\infty} \xi_n = \infty$ and $ \lim_{n\to\infty}
   \zeta_n < \infty$, then
   \begin{gather*}
     \lim_{n\to\infty} (x+\xi_n)/(y+\zeta_n) = \infty =
     \lim_{n\to\infty} \xi_n/\zeta_n  \; .
   \end{gather*}
  \end{enumerate}

\end{proof}

\begin{proof}[Proof of~(\ref{eq:equiv-b})] 
  By assumption, $\pr(Y>y) \leq \bar H(y/v^*)$, thus $b_Y(t) \leq
  v^*b(t)$.  Denote $V=v(T)$.  Two cases only are possible: (i) $v^*$ is
  an isolated point of the support of the distribution of $V$ and
  $\pr(V=v^*)>0$; (ii) $\pr(V=v*)=0$ and for any $\epsilon>0$, there
  exists $v\in(v^*-\epsilon,v^*)$ such that $\pr(V>v)>0$.
  \begin{enumerate}[(i)]
  \item If $\pr(V=v^*)>0$ and $v^*$ is an isolated point, then there
    exists $v^{**}<v^*$ such that 
    \begin{gather*}
      \pr(Y>y) = \bar H(y/v^*) \pr(V=v^*) + \pr(RV>y \;, \; V \leq
      v^{**}) \; .
    \end{gather*}
    Note that $\pr(RV>y \;, \; V \leq v^{**}) \leq \bar H(y/v^{**})$,
    and since $\bar H$ is $\Gamma$-varying, $\bar H(y/v^{**})=o(\bar
    H(y/v^*))$, thus $\pr(Y>y) \sim \epsilon\bar H(y/v^*)$ with
    $\epsilon = \pr(V=v^*)$. Since $b$ is slowly varying, this implies
    that $b_Y(t) \sim v^* b(\epsilon t)$ and finally that $b_Y(t) \sim
    v^* b(t)$ as $t\to\infty$.
  \item In the second case, fix some $v<v*$ and denote
    $\epsilon=\pr(V>v)>0$ by assumption. Then 
    \begin{align*}
      \pr(Y>y) & = \pr(RV>y \;, \; V > v) + \pr(RV>y \;, \; V \leq
      v) 
       \geq \epsilon \pr(R>y/v). 
    \end{align*}
    Thus, $b_Y(t) \geq v b(\epsilon t)$ and since $b$ is slowly varying, 
    \begin{gather*}
      \liminf_{t\to\infty} \frac{b_Y(t)}{b(t)} \geq v \; .
    \end{gather*}
    Since $v$ can be chosen arbitrarily close to $v^*$ and since we
    already know that $b_Y(t) \leq v^*b(t)$, we conclude that
    $\lim_{t\to\infty} b_Y(t)/b(t) = v^*$.
  \end{enumerate}

\end{proof}

\paragraph{Acknowledgement} 
We thank the associate editor and the referees for their comments and
suggestions that helped to substantially improve our paper.

\bibliographystyle{plainnat} \bibliography{bibextreme}

\begin{thebibliography}{22}
\providecommand{\natexlab}[1]{#1}
\providecommand{\url}[1]{\texttt{#1}}
\expandafter\ifx\csname urlstyle\endcsname\relax
  \providecommand{\doi}[1]{doi: #1}\else
  \providecommand{\doi}{doi: \begingroup \urlstyle{rm}\Url}\fi

\bibitem[Abdous et~al.(2005)Abdous, Foug\`eres, and
  Ghoudi]{abdous:fougeres:ghoudi:2005}
Belkacem Abdous, Anne-Laure Foug\`eres, and Kilani Ghoudi.
\newblock Extreme behaviour for bivariate elliptical distributions.
\newblock \emph{Revue Canadienne de Statistiques}, 33\penalty0 (2):\penalty0
  1095--1107, 2005.

\bibitem[Abdous et~al.(2008)Abdous, Foug\`eres, Ghoudi, and
  Soulier]{abdous:fougeres:ghoudi:soulier:2008}
Belkacem Abdous, Anne-Laure Foug\`eres, Kilani Ghoudi, and Philippe Soulier.
\newblock Estimation of bivariate excess probabilities for elliptical models.
\newblock \emph{Bernoulli}, 14\penalty0 (4):\penalty0 1065--1088, 2008.

\bibitem[Balkema and Embrechts(2007)]{balkema:embrechts:2007}
Guus Balkema and Paul Embrechts.
\newblock \emph{{High risk scenarios and extremes. A geometric approach.}}
\newblock {Zurich Lectures in Advanced Mathematics. Z\"urich: European
  Mathematical Society}, 2007.

\bibitem[Barbe(2003)]{barbe:2003}
Philippe Barbe.
\newblock Approximation of integrals over asymptotic sets with applications to
  probability and statistics.
\newblock http://arxiv.org/abs/math/0312132, 2003.

\bibitem[Berman(1983)]{berman:1983}
Simeon~M. Berman.
\newblock Sojourns and extremes of {F}ourier sums and series with random
  coefficients.
\newblock \emph{Stochastic Processes and their Applications}, 15\penalty0
  (3):\penalty0 213--238, 1983.

\bibitem[Berman(1992)]{berman:1992}
Simeon~M. Berman.
\newblock \emph{Sojourns and extremes of stochastic processes}.
\newblock The Wadsworth \& Brooks/Cole Statistics/Probability Series. Wadsworth
  \& Brooks/Cole Advanced Books \& Software, Pacific Grove, CA, 1992.

\bibitem[Bingham et~al.(1989)Bingham, Goldie, and
  Teugels]{bingham:goldie:teugels:1989}
N.~H. Bingham, C.~M. Goldie, and J.~L. Teugels.
\newblock \emph{Regular variation}, volume~27 of \emph{Encyclopedia of
  Mathematics and its Applications}.
\newblock Cambridge University Press, Cambridge, 1989.

\bibitem[Breiman(1965)]{breiman:1965}
Leo Breiman.
\newblock On some limit theorems similar to the arc-sin law.
\newblock \emph{Theory of Probability and its Applications}, 10:\penalty0
  351--360, 1965.

\bibitem[Cline and Samorodnitsky(1994)]{cline:samorodnitsky:1994}
Daren B.~H. Cline and Gennady Samorodnitsky.
\newblock Subexponentiality of the product of independent random variables.
\newblock \emph{Stochastic Process. Appl.}, 49\penalty0 (1):\penalty0 75--98,
  1994.

\bibitem[Das and Resnick(2008)]{das:resnick:2008}
Bikramjit Das and Sidney~I. Resnick.
\newblock Conditioning on an extreme component: Model consistency and regular
  variation on cones.
\newblock http://arxiv.org/abs/0805.4373, 2008.

\bibitem[Das and Resnick(2009)]{das:resnick:2009}
Bikramjit Das and Sidney~I. Resnick.
\newblock Detecting a conditional extrme value model.
\newblock http://arxiv.org/abs/0902.2996, 2009.

\bibitem[De~Haan and Ferreira(2006)]{dehaan:ferreira:2006}
Laurens De~Haan and Ana Ferreira.
\newblock \emph{{Extreme value theory. An introduction.}}
\newblock {Springer Series in Operations Research and Financial Engineering.
  New York, NY: Springer.}, 2006.

\bibitem[Eddy and Gale(1981)]{eddy:gale:1981}
William~F. Eddy and James~D. Gale.
\newblock The convex hull of a spherically symmetric sample.
\newblock \emph{Advances in Applied Probability}, 13\penalty0 (4):\penalty0
  751--763, 1981.

\bibitem[Foug\`eres and Soulier(2008)]{fougeres:soulier:2008}
Anne-Laure Foug\`eres and Philippe Soulier.
\newblock Estimation of conditional laws given an extreme component.
\newblock arXiv:0806.2426, 2008.

\bibitem[Hashorva(2005)]{hashorva:2005}
Enkelejd Hashorva.
\newblock Extremes of asymptotically spherical and elliptical random vectors.
\newblock \emph{Insurance: Mathematics \& Economics}, 36\penalty0 (3):\penalty0
  285--302, 2005.

\bibitem[Hashorva(2006)]{hashorva:2006}
Enkelejd Hashorva.
\newblock Gaussian approximation of conditional elliptic random vectors.
\newblock \emph{Stoch. Models}, 22\penalty0 (3):\penalty0 441--457, 2006.

\bibitem[Hashorva(2008)]{hashorva:2008polar}
Enkelejd Hashorva.
\newblock Conditional limit results for type {I} polar distributions.
\newblock Extremes, 10.1007/s10687-008-0078-y, 2008.

\bibitem[Hashorva et~al.(2007)Hashorva, Kotz, and
  Kume]{hashorva:kotz:kume:2007}
Enkelejd Hashorva, Samuel Kotz, and Alfred Kume.
\newblock {$L\sb p$}-norm generalised symmetrised {D}irichlet distributions.
\newblock \emph{Albanian Journal of Mathematics}, 1\penalty0 (1):\penalty0
  31--56 (electronic), 2007.

\bibitem[Heffernan and Resnick(2007)]{heffernan:resnick:2007}
Janet~E. Heffernan and Sidney~I. Resnick.
\newblock {Limit laws for random vectors with an extreme component.}
\newblock \emph{Annals of Applied Probability}, 17\penalty0 (2):\penalty0
  537--571, 2007.

\bibitem[Heffernan and Tawn(2004)]{heffernan:tawn:2004}
Janet~E. Heffernan and Jonathan~A. Tawn.
\newblock A conditional approach for multivariate extreme values.
\newblock \emph{Journal of the Royal Statistical Society. Series B},
  66\penalty0 (3):\penalty0 497--546, 2004.

\bibitem[Resnick(2002)]{resnick:2002}
Sidney Resnick.
\newblock Hidden regular variation, second order regular variation and
  asymptotic independence.
\newblock \emph{Extremes}, 5\penalty0 (4):\penalty0 303--336, 2002.

\bibitem[Resnick(1987)]{resnick:1987}
Sidney~I. Resnick.
\newblock \emph{Extreme values, regular variation, and point processes},
  volume~4 of \emph{Applied Probability. A Series of the Applied Probability
  Trust}.
\newblock Springer-Verlag, New York, 1987.

\end{thebibliography}

\end{document}